\documentclass{amsart}

\usepackage{amsmath}

\usepackage{amssymb}
\usepackage{amscd}

\numberwithin{equation}{section}

\newif\ifdraft\drafttrue

\draftfalse

\newcommand{\br}{{\mathbb{R}}}

\newcommand\thz{\theta_{H}}
\newcommand\zt{Z_{[0,\tau]}}

\font\sb = cmbx8 scaled \magstep0

\font\sn = cmssi8 scaled \magstep0

\long\def\comdima#1{\ifdraft{\marginpar{\sb
#1 \ DK}}\else\ignorespaces\fi}

\newcommand\ba{badly approximable}

\newcommand\hs{homogeneous space}

\newcommand\ssm{\smallsetminus}

\newcommand\name[1]{\label{#1}{\ifdraft{\sn [#1]}\else\ignorespaces\fi}}

\newcommand\eq[2]{{\ifdraft{\ \tt [#1]}\else\ignorespaces\fi}\begin{equation}\label{eq:#1}{#2}\end{equation}}

\newcommand {\equ}[1]     {\eqref{eq:#1}}

\newcommand{\goth}[1]{{\mathfrak{#1}}}

\newcommand\g{\goth g}
\newcommand\h{\goth h}

\newcommand\f{\goth f}

\newcommand{\R}{{\mathbb{R}}}

\newcommand{\Z}{{\mathbb{Z}}}

\newcommand{\N}{{\mathbb{N}}}

\newcommand{\vv}{{\bf{v}}}

\newcommand{\w}{\mathbf{w}}

\newcommand{\Exp}{{\operatorname{Exp}}}

\newcommand{\Ad}{{\operatorname{Ad}}}

\newcommand{\Lie}{\operatorname{Lie}}

\newcommand{\SL}{\operatorname{SL}}

\newcommand{\ggm}{G/\Gamma}

\newcommand{\dist}{\operatorname{dist}}

\newcommand {\ignore}[1]  {}

\newcommand{\df}{{\, \stackrel{\mathrm{def}}{=}\, }}

\newcommand{\x}{{\mathbf{x}}}

\newcommand{\ve}{{\bf e}}

\newcommand{\p}{{\bf p}}

\newcommand{\til}{\widetilde}

\newcommand{\sm}{\smallsetminus}

\newcommand{\vre}{\varepsilon}

\font\comment = cmbx10 scaled \magstep0

\newcommand\hd{Hausdorff dimension}

\newcommand\nz{\smallsetminus \{0\}}

\newcommand{\cl}{\mathcal L}

\newtheorem{thm}{Theorem}[section]
\newtheorem{lem}[thm]{Lemma}
\newtheorem{prop}[thm]{Proposition}
\newtheorem{cor}[thm]{Corollary}

\newtheorem{remark}[thm]{Remark}

\begin{document}

\title[Binary quadratic forms and games]{Values of binary quadratic
  forms at integer points and Schmidt games}
\author[ Kleinbock and Weiss]{ Dmitry Kleinbock and  Barak Weiss}

\address{Department of Mathematics, Brandeis University, Waltham MA 02454-9110, USA} 
\email{kleinboc@brandeis.edu}

\address{Department of Mathematics, Tel Aviv University, Tel Aviv, Israel} 
\email{barakw@post.tau.ac.il}

\date{September 19, 2013}


\begin{abstract}
We prove that for any countable set $A$ of real numbers, the set of
binary indefinite quadratic forms   $Q$ such that the closure of
$Q\left(\Z^2\right)$ is disjoint from $A$ has full \hd.   
\end{abstract}

\maketitle{}

\centerline{\it Dedicated to S.G.\ Dani on the occasion of his 65th birthday}

\medskip

\section{Introduction}

We start with the following statement, conjectured by Oppenheim and
Davenport in the 1930s and proved by Margulis in the 1980s: 

\begin{thm}\name{opp} Let $Q$ be a real  nondegenerate indefinite
quadratic form  of
$n > 2$ variables which is not proportional to a
rational form. Then  \eq{bad}{0\text{ belongs to the closure of }Q\big(\Z^n\nz\big).} 
\end{thm}

Margulis used an approach which was suggested earlier by
Raghunathan and implicitly used earlier by Cassels and Swinnerton-Dyer \cite{CSD}. Let $Q_0$ be a fixed quadratic form of the same signature
as $Q$, then one can write $Q(\vv) = aQ_0(g\vv)$ for some $g\in G =
\SL_n(\R)$ and $a \in \R$. Let $F$ be the stabilizer of $Q_0$ and let  $\Gamma
=\SL_n(\Z)$. Then   \equ {bad} holds if and only if the orbit $Fg\Z^n$
in the space $\ggm$ of unimodular lattices in $\R^n$ is unbounded. The
theorem proved by Margulis (in the case $n = 3$, to which the general
case can be reduced) stated that any bounded orbit must be compact,
from which it is not hard to derive Theorem \ref{opp}. 

Note that  it was later proved by Dani and Margulis
that for $Q$ as above the set $Q\big(\Z^n\big)$, and even
$Q\big(P(\Z^n)\big)$, is dense in $\R$ (here $P(\Z^n)$ stands for the
set of primitive integer points in $\Z^n$).   

When $n=2$ it is well known that conclusion of Theorem \ref{opp}
fails. Namely, take 
\eq{binary form}{Q(p,q) = p^2 - \lambda^2 q^2 =
  {q^2}\left(\frac pq - \lambda\right)\left(\frac pq + \lambda\right)} 
such that   $\lambda$ is {\sl badly approximable\/}, that
is, 
$$\inf_{p\in\Z,\,q\in\N}q^2\left|\frac pq - \lambda\right| >
0\,;$$ 
then the absolute value of $Q(p,q)$ is uniformly bounded away
from $0$ for any nonzero integer $(p,q)$. It is known from the work of
Jarn\'ik \cite{J} that the set of such $\lambda$'s, although null, is
{\sl thick\/}, that is, its intersection with any nonempty open set
has full \hd. 
As in the reduction of the Oppenheim conjecture to a dynamical
statement, choose $Q_0(x,y) \df xy$, let  
\eq{defG}{G\df\SL_2(\R)\,,}
and let $F$ be the connected
component of the identity in the stabilizer of $Q_0$, 
namely 
\eq{defF}{F =
  \{g_t : t\in \R\}\text{ where }g_t = \begin{pmatrix}e^t & 0\\ 0 &
    e^{-t}\end{pmatrix}\,.} We also let 
\eq{defX}{X \df \ggm \text{
    where }\Gamma \df\SL_2(\Z)\,, \text{  and } \mathcal{Q} \df F \backslash G;} 
thus $X$ can be identified with the space of unimodular lattices in
$\R^2$, and $\mathcal{Q}$ can be identified with the space of binary
indefinite quadratic forms, considered up to scaling. The set of
integer values of a quadratic form is a (set-valued) function on $F
\backslash G / \Gamma$ but since this double coset space has a
complicated topological structure, it is more useful to consider it
either as an $F$-invariant function on $X$, or dually, as a
$\Gamma$-invariant function on $\mathcal{Q}$. This duality principle
(already evident in the work of Cassels and Swinnerton-Dyer \cite{CSD}
and developed explicitly by Dani \cite{Da2} in a related context) makes it
possible to recast dynamical properties of the $F$-action on $X$ as
properties of quadratic forms. In particular (see e.g.\ \cite[Lemma 2.2.1]{Margulis Oppenheim}),
it follows from Mahler's Compactness Criterion that for a quadratic
form $Q(\vv) = Q_0(g\vv) \in \mathcal{Q}$, the set of 
values $Q\left(\Z^2\right)$ has a gap at zero if and only if the orbit $F(g\Z^2)$ is
bounded in $X$. 
In light of this Jarn\'ik's theorem
implies:

 \begin{thm}\name{ba}  Let 
 $X$ and  $Q_0$ 
 be as above. 
 Then 
 \ignore{ 
\begin{itemize}\item [\rm (i)] 
The set
$
\big\{x\in X : Fx \text{ is bounded}  \}$
is  thick; 
\item [\rm (ii)] }
the set
$$
\left\{x\in X : 0 \notin \overline{Q_0(x\nz)} \right\}$$ is thick;
dually, the set of binary indefinite
quadratic forms whose set of values at integer points has a gap at $0$ 
is  thick in the space of all binary indefinite forms.
\end{thm}

It is worth pointing out that a
generalization of this argument  is due to S.G.\ Dani \cite{Da2, Dani-rk1}, who used W.\ Schmidt's results and methods \cite{S1, S2} to
find thick sets of points with bounded orbits in other \hs s. See also
\cite{Ar, KM, KW3} for further developments.  


\smallskip
The goal  of this paper is to strengthen Theorem \ref{ba} by
considering even more complicated properties of the sets of values of
quadratic forms at integer points, and, correspondingly, the behavior
of $F$-orbits in $X$. Here is one of our main results:

 \begin{thm}\name{ba countable}  For any countable subset $A\subset \R$, the set
$$
\big\{x\in X :  \overline{Q_0(x\nz)}\cap A = \varnothing \}$$ is
thick. Consequently, the set of binary indefinite quadratic forms
whose set of values at integer points miss a given countable set 
is  thick in the space of all binary indefinite forms.
 \end{thm}
 
 Note that by a theorem of Lekkerkerker \cite{Lek}, for $Q$ as in
 \equ{binary form} the set of accumulation points of $Q\left(\Z^2 \right)$ is
 discrete if and only if $\lambda$ is quadratic irrational. See
 \cite{TV, DN} for a precise description of the sets of accumulation
 points in these cases. 

\smallskip
We will derive Theorem \ref{ba countable}  from its more general
dynamical counterpart. To state it, we need the following
definition. Let $H$ be a 
connected subgroup of $G$ different from $F$, 
and let $Z$ be a submanifold of $X$. Say that $Z$ is {\sl
  $(F,H)$-transversal\/} if the following two conditions hold:  

\smallskip
\begin{itemize}
\item[$(F)$] for any $x\in Z$, $T_x(Fx)$ is not contained in $T_xZ$;
\medskip
\item[$(H,F)$] for any $x\in Z$, $T_x(Hx)$ is not contained in $T_xZ \oplus T_x(Fx)$.
\end{itemize}

\medskip

For example, the above conditions are satisfied if $Z$ consists of a
single point. If $\dim(Z) = 1$, the $(F,H)$-transversality of $Z$ is
equivalent to saying that for each $x\in Z$ the lines $T_x(Hx)$,
$T_xZ$ and $T_x(Fx)$ are in general position, i.e.\ they generate the
space $T_xX$.

We permit $Z$ to be a manifold with boundary; in such a
case, smoothness of maps and 
definitions of tangent spaces at the boundary points are defined by
positing the existence of smooth extensions (see
e.g.\ \cite[Chap.\ 2]{GP}). We also note that in the application to
quadratic forms, $Z$ is an analytic manifold, but in all our
arguments, it suffices to assume that $Z$ is $C^1$-smooth.

\smallskip
Now let us denote by $H^+$ and $H^-$ the {\sl expanding and contracting horospherical subgroups\/} with respect to $g_1$:
\eq{horosph}{H^+\df\left\{h_s\df \begin{pmatrix}1  & s\\ 0 & 1 \end{pmatrix} : s\in\R\right\}, \quad H^-\df\left\{\begin{pmatrix}1  & 0\\ s & 1 \end{pmatrix} : s\in\R\right\}\,.}

A dynamical statement which will imply Theorem \ref{ba countable} is as follows:

 \begin{thm}\name{dyn countable}  Let $Z$ be a  countable union of 
  submanifolds  of $X$ which are both $(F,H^+)$- and $(F,H^-)$-transversal. Then  the set
\eq{result}{
\big\{x\in X : Fx \text{ is bounded and } \overline{Fx}\cap Z  = \varnothing\}}
is  thick.
 \end{thm}

Note that for arbitrary \hs s $X = \ggm$, non-quasiunipotent subgroups $F$ and {\it finite\/} sets $Z$ the thickness of the set \equ{result}
was conjectured by Margulis \cite{Ma}. Then in \cite{KM} it was shown
that, for mixing flows,  the set of points with bounded orbits is
thick, and  in    \cite{K}, for arbitrary \hs s and $Z$ being a {\it
  finite\/} union of $Z_i$ as in the above
theorem\footnote{Technically the statement in \cite[Corollary
  4.4.2]{K} is weaker than quoted here, namely $Z_i$ are assumed to
  satisfy condition $(F)$ and have dimension smaller than dimensions
  of $H^+$ and $H^-$ -- but the argument relies precisely on the
  combination of $(F,H^+)$-  and $(F,H^-)$-transversality.} 
--  that the set 
$$
\big\{x\in X :  \overline{Fx}\cap Z  = \varnothing\}
$$
is thick. Margulis' Conjecture was finally proved in \cite{KW3} using
the technique of Schmidt games, with the set $Z$ upgraded from finite
to countable. However the argument of  \cite{KW3} could not produce a
result for $Z$ being more than zero-dimensional, which, in particular,
is needed for an application to quadratic forms.  

\smallskip

{\em Organization of the paper.} In \S \ref{sec: gaps} we 
explain the reduction of Theorem
\ref{ba countable} to Theorem \ref{dyn countable}. In \S \ref{games} we describe a variant of Schmidt's $(\alpha, \beta)$-game. This variant is very close to the {\em absolute game} and {\em hyperplane absolute game}
introduced in \cite{Mc} and \cite{BFKRW} respectively, but
adapted to a situation in which we want to play on a smooth
manifold. In \S \ref{slicing} we state our main technical
result, Theorem \ref{h discrete winning} which implies Theorem
\ref{dyn countable}, and is a
result on the winning property of a certain set for 
the above game. We complete the proof of Theorem \ref{h discrete winning} in 
\S \ref{sec: completing proof}. 

\smallskip

{\em Acknowledgements.} The authors gratefully acknowledge the support
of BSF grant 2010428, ERC starter grant DLGAPS 279893, and NSF grant DMS-1101320. 

\section{Dynamics and gaps between values of quadratic
  forms}\name{sec: gaps}
In this section we explain how to reduce Theorem  \ref{ba countable}
to Theorem \ref{dyn countable}. As was mentioned in the 
introduction,  the equivalence between the set of values of the form
being bounded away from $0$ and the $F$-orbit of the corresponding
lattice  being bounded in $X$ is well-known. We need to understand how
to dynamically describe the set of quadratic forms whose values miss a
fixed $a\ne 0$. So fix a nonzero $a$ and choose $\vv\in\R^2$ such that
$Q_0(\vv) = a$. Not much will depend on this choice, yet the most
obvious one seems to be choosing \eq{defy}{\vv
  = \begin{cases}(\sqrt{a}, \sqrt{a}) \text{\ \ \  if } a >
    0\\(-\sqrt{a}, \sqrt{a}) \text{ otherwise } \end{cases}} which is
what we will do. Now define 
$$
\tilde Z_\vv \df \{x\in X : \vv\in x\}\,,
$$
that is, the set of unimodular lattices in $\R^2$ containing
$\vv$. The structure of this set is easy to describe: $ 
\tilde Z_\vv = \bigcup_{n\in \N}Z_{\frac{1}{n}\vv}$ where 
$$
 Z_\vv \df \{x\in X : \vv\in P(x)\}\,,
$$
the set of unimodular lattices in $\R^2$ containing $\vv$ as a {\it
  primitive\/} vector. The latter is simply a closed horocycle, namely
a periodic  orbit of the subgroup $V$ of $G$ stabilizing $\vv$, which
is easily seen to be, for our choice of $\vv$, equal to 
\eq{defu}{V \df  \begin{cases}\left\{\begin{pmatrix}1+s & -s\\ s &
        1-s \end{pmatrix} : s\in\R\right\}\text{\ \ \  \ \ if } a >
    0\\ \quad \\ \left\{\begin{pmatrix}-1-s & s\\ s &
        -1+s \end{pmatrix} : s\in\R\right\}\text{ otherwise
    } \end{cases}\,. 
}
 Now let us state a proposition which relates escaping $
\tilde Z_\vv$ with a gap in values of quadratic forms at $a$:

 \begin{prop}\name{tlink}  Suppose $x\in X$ is such that $Fx$ is
   bounded, let $a \neq 0$, and define $\vv$ by \equ{defy}. Then  $
   \overline{Fx}\cap \tilde Z_\vv  = \varnothing$ if and only if  $a\notin
   \overline{Q_0(x\nz)}$. \end{prop} 
 
 \begin{proof} 
For the `if' direction, suppose that $x_0 \in \overline{Fx} $ contains
$\vv$. Then there are $t_n$ such that $g_{t_n}x \to x_0$ and so there
are $\w_n \in x$ with $g_{t_n}\w_n \to \vv$. In particular $\w_n \neq
0$. Hence
$Q_0(\w_n)=Q_0(g_{t_n}\w_n) \to Q_0(\vv)=a$, so $a \in \overline{Q_0(x \sm
\{0\})}$, a contradiction. 
 
Suppose the converse does not hold, that is, there
   exist vectors $\vv_k\in x$ such that $Q_0(\vv_k) \to a$ as
   $k\to\infty$. For each $k$, choose $t_k\in \R$ such that
   $g_{t_k}\vv_k$ belongs to the line passing through $\vv$; then it is
   easy to see that 
\eq{conv}{g_{t_k}\vv_k\to \vv\text{ as
     }k\to\infty\,.} 
But $Fx$ is relatively compact, hence there
   exists a limit point $x_0$ of the sequence $g_{t_k}x$ of lattices,
   and from \equ{conv} it follows that $x_0$ contains $\vv$, or,
   equivalently, $x_0\in  \tilde Z_\vv$, contrary to 
   assumption. 
 \end{proof}
 
\begin{remark} {\rm
The boundedness of $Fx$ was only used in proving the implication
$\implies$. If $Fx$ is not bounded (or equivalently, $0 \in
\overline{Q_0(x \sm \{0\})}$) then the situation is more 
interesting. A simple argument involving multiplication by integers,
shows that when $Q_0(x \sm \{0\}) $ contains sequences approaching $0$
from both sides, then the set of values $Q_0(x \sm \{0\})$ is
dense. However, as shown by Oppenheim \cite{Oppenheim}, lattices $x$
for which $0$ is only a one-sided limit of $Q_0(x \sm \{0\})$ do
exist. }
\end{remark}

We record the following:

 \begin{lem}\name{transversality}  For  $\vv$ as in \equ{defy} (in fact, the same is true for every $\vv$ with $Q_0(\vv) \ne 0$), the manifold
 $Z_\vv$ is both $(F,H^+)$-transversal and $(F,H^-)$-transversal. \end{lem}
 
 \begin{proof} It suffices to show that the Lie algebras of $F$, $V$
   and $H^+$, as well as those of $F$, $V$ and $H^+$, span the Lie
   algebra of $G$. This is an easy computation  using \equ{defF},
   \equ{horosph} and \equ{defu}.\end{proof}

 We can now see that Theorem \ref{ba countable}  follows from Theorem \ref{dyn countable}:

 \begin{proof}[Proof of Theorem \ref{ba countable} assuming Theorem \ref{dyn countable}]  
 We have already mentioned that the boundedness of the $F$-orbit of
 $x$ implies that $0$ is not in the closure of the set
 $Q_0(x\nz)$. Now for each $a\in A\ssm 0$ take $\vv = \vv(a)$ as in
 \equ{defy}, then, in view of Lemma \ref{transversality}, the set $$ Z
 = \bigcup_{a\in A\nz}\tilde Z_{\vv(a)}$$ satisfies the assumption of
 Theorem \ref{dyn countable}, hence the set \equ{result} is thick. 
 On the other hand, Proposition \ref{tlink} implies that  for every
 $x$ from the set, $ \overline{Q_0(x\nz)}\cap A = \varnothing$.
 \end{proof}

\section{Schmidt games}\name{games}
We first recall Schmidt's $(\alpha, \beta)$-game, introduced in
\cite{S1}. The game is played by two players Alice and Bob on a
complete metric space\footnote{In this section, following tradition,
  we denote this metric space by $X$; elsewhere $X$ continues
  to denote the space of two-dimensional unimodular lattices.} $X$ equipped 
with a target set $S$ and two fixed parameters $\alpha, \beta \in
(0,1)$. Bob begins the $(\alpha,\beta)$-game
by choosing $x_1 \in X, r_1>0$, thus specifying the closed ball $B_1
\df \bar{B}(x_1, r_1)$, where 
$$\bar{B}(z,\rho)  \df \{x \in X: \dist(x,z) \leq \rho\}.$$
Then Alice and Bob continue by alternately  
choosing $x'_i,\, x_{i+1}$ so that 
$$
\dist(x_i, x'_i) \leq (1-\alpha)r_i, \ \  \dist(x'_{i+1}, x_i) \leq (1-\beta)r'_i ,
\ \ \text{ where } \  
r'_i \df  \alpha r_i, \ \  r_{i+1} \df  \beta r'_i.
$$
This implies that 
the closed balls 
$$
A_i \df  \bar{B} (x'_i, r'_i), \ \ B_{i+1} \df
\bar{B}(x_{i+1}, r_{i+1})
$$
are nested, i.e. 
$$B_1 \supset A_1 \supset B_2 \supset\cdots$$
The set $S$ is said to be 
 {\sl
  $\alpha$-winning\/} if for any
$\beta>0$ Alice has a strategy in the  $(\alpha,\beta)$-game guaranteeing that  
the unique point of intersection 
$\bigcap_{i=1}^\infty B_i= \bigcap_{i=1}^\infty A_i$ of all the balls
belongs to $S$,
regardless of the way Bob chooses to play. It is called {\sl
  winning\/} if it is $\alpha$-winning for some  
$\alpha$. 

\smallskip
In \cite{BFKRW}, 
inspired by ideas of McMullen \cite{Mc}, the
{\em absolute hyperplane game} was introduced. This modification is 
played on $\R^d$. Let $S \subset \R^d$ be a target set and let 
$\beta \in \left(0, \frac13 \right)$. 
As 
before Bob begins by choosing a closed ball $B_1$ of radius $r_1$ and
then Alice and Bob alternate moves. The sets $B_i$ chosen by Bob are
closed balls of radii $r_i$ 
satisfying 
$$r_{i+1} \geq \beta r_{i}\,.$$
The sets 
$A_i$ chosen by Alice are $r'_i$-neighborhoods of affine
hyperplanes, where the $r'_i$ satisfy 
$r'_i \leq \beta r_i. $
Additionally Bob's choices must satisfy 
$$B_{i+1} \subset B_i \sm A_i\,.$$
Then $S \subset \R^d$ is said to be {\sl $\beta$-HAW\/} (where HAW is
an acronym for {\em hyperplane absolute winning}) if
Alice has a strategy which leads to 
\eq{Alicewins}{\bigcap_{i=1}^\infty B_i\cap S\ne\varnothing} 
regardless of how Bob chooses to play; $S$ is said to be  {\sl
  HAW\/} if it is $\beta$-HAW for all $0
  <\beta < \frac13$. It is easy to see that $\beta$-HAW implies
  $\beta'$-HAW whenever $\beta \le \beta' < 1/3$; thus a set is HAW
  iff it is $\beta$-HAW for arbitrary
  small positive $\beta$. In the case $d = 1$ hyperplanes are points,
  and thus the HAW property coincides with the
  {\sl absolute winning\/}
  property\footnote{More generally, the paper \cite{BFKRW}  introduced
    $k$-dimensional absolute winning for every $0 \le k < d$; the case
    $k=0$ was considered earlier by McMullen.} introduced by McMullen
  in  \cite{Mc}. 

\ignore{
The parameter $\beta_0$ is chosen so as to
  guarantee that at every stage of the game, there are legal moves for
  Bob; i.e. that for any sequence of choices $B_1, A_1, \ldots, B_i,
  A_i$, there is a closed ball $B_{i+1}$ in $X$ with radius satisfying
  \equ{eq: Bobs choice} which
  also satisfies \equ{eq: Bobs choice 2}.  It is easy to see that
  $\beta_0 = 1/3$ works for $X = \R^d$, 
  and in \cite{HAW} several sufficient conditions on a set $X$ were
  given, in terms of measures supported on $X$, guaranteeing the
  existence of $\beta_0$ with this property. 

}
\smallskip
The following proposition summarizes properties of winning 
and HAW subsets of $\R^d$:
\begin{prop}\name{properties}
\begin{itemize}
\item[(a)] Winning sets are thick.
\item [(b)] HAW implies winning. 
\item[(c)]  The countable intersection of $\alpha$-winning (resp., HAW) sets is again $\alpha$-winning (resp., HAW). 
\item  [(d)]
The  image of a HAW 
set under a $C^1$ diffeomorphism $\R^d \to \R^d$ is HAW.
\end{itemize}\end{prop}
For the proofs, see \cite{S1, Mc, BFKRW}.

\begin{remark}\name{remark: radii} {\rm
Note that in the hyperplane absolute version the radii of balls do not have to tend to zero,
therefore $\cap_i B_i$ 
does not have to be a single point. However 
the outcome with radii not tending to $0$ is clearly winning for Alice
as long as $S$ is dense. Thus in all the proofs of the HAW property it
will be safe to assume that Bob plays so that $r_n\to 0 $ as
$n\to\infty$: indeed, if Alice has a strategy which is guaranteed to
win the game whenever $r_n \to 0$, then the
target set must be dense, and hence the strategy is guaranteed to work
even if $r_n \not \to 0$.}
\end{remark}

We will be interested in playing variants of the two games described
above on a differentiable manifold. Note that a manifold is not
equipped with an intrinsic metric, nor with an intrinsic notion of
affine submanifolds, and thus the definitions given above cannot be
applied directly. Our approach will be to work in a given coordinate
system and argue using Proposition \ref{properties} that the class of winning
sets does not depend on the choice of a coordinate system. It will be
technically simpler to work with the hyperplane absolute game and this
is all that we require for the present paper. We proceed to
the details. 

We first define the absolute hyperplane game on an open subset $W
\subset \R^d$. This is defined just as the absolute hyperplane game on
$\R^d$, except for requiring that Bob's first move $B_1$ be
contained in $W$. If Alice has a winning strategy, we say that $S$ is
{\em HAW on $W$}. Now 
suppose $X$ is a $C^1$ $d$-dimensional manifold equipped with an atlas of
charts $\mathcal{A} = (U_\alpha, \varphi_\alpha)$; that is, $X$ is a separable topological
space, the sets $U_\alpha$ are open subsets of $X$ with $X =
\bigcup_\alpha U_\alpha$, each $\varphi_\alpha: U_\alpha \to \R^d$ is
a homeomorphism onto its image, and the transition functions
$\varphi_\beta \circ \varphi_\alpha^{-1}: \varphi_\alpha(U_\alpha \cap
U_\beta) \to \R^d$ are $C^1$. To define the absolute hyperplane
game on $(X, \mathcal{A})$, we specify a target set $S \subset X$.  
 Bob begins play by choosing one coordinate chart
$(U, \varphi) = (U_\alpha, \varphi_\alpha)$ in $\mathcal{A}$ and a closed ball $B_1 \subset \R^d$
contained in $\varphi (U)$. From this point on the game continues as
before, where Alice and Bob alternate moves in $\varphi(U) \subset
\R^d$. To decide the winner they check whether the point of
intersection $\bigcap_i B_i$ belongs to $\varphi(S)$. If Alice has a
winning strategy, we say that $S$ is {\em HAW on $(X, \mathcal{A})$.}

Note that when $W$ is an open subset of $\R^d$, the definition of the
game on $W$, given at the beginning of the previous paragraph,
coincides with the definition of the game on $(X, \mathcal{A})$ if we
take 
$X=W$ and take $\mathcal{A}$ to be the atlas consisting of one chart $(W,
\mathrm{Id}).$ 
Also 
note that Bob has been given the additional prerogative of choosing a
coordinate chart within which to work, at the start of play, and this
appears to make the winning property very restrictive. However we have: 

\begin{prop}\name{prop: does not depend on chart}
Suppose $X$ is a $C^1$  manifold with an atlas $\mathcal{A}$ and $(U_i, \varphi_i)$ is a 
system of coordinate charts 
in $\mathcal{A}$, such that $X = \bigcup U_i$, and $S \subset X$. Then $S$ is
HAW on $(X, \mathcal{A})$ if and only if for 
each $i$, $\varphi_i(S)$ is HAW  on
$\varphi_i(U_i)$. 
\end{prop}
\begin{proof}
The direction $\Longrightarrow$ is immediate from the definitions,
since Bob may select to work with each of the charts $\varphi_i$ on
his first move. For
the direction $\Longleftarrow$, note first that by Proposition
\ref{properties}(a,b), $\varphi_i(S)$ is dense in each $\varphi_i(U_i)$
and hence $S$ is dense in $X$. 
Now suppose Bob chose to work in some
chart $\varphi_\alpha$ distinct from the $\varphi_i$. If the diameters
of the balls
$B_n$ chosen by Bob do not decrease to zero, then $\bigcap B_n$ has
interior. Since
$S$ is dense in $X$, each $\varphi_\alpha(S)$ is dense in
$\varphi_\alpha(U_\alpha)$, so Alice wins. Otherwise, since $B_1$ is compact and is
covered by the open sets $\varphi_\alpha(U_i)$, there is a Lebesgue
number for this cover, that is $\delta>0$ such that each subset of
$B_1$ of diameter at most $\delta$ is contained in one of the
$\varphi_\alpha(U_i)$. Thus there are $n, i$ such
that 
\eq{eq: containment}{\varphi_{\alpha}^{-1} (B_n) \subset U_i.}

In light of Proposition 
\ref{properties}(d), $\varphi_\alpha(S) = \varphi_\alpha \circ
\varphi^{-1}_i \big(\varphi_i(S)\big)$ is HAW
on $\varphi_\alpha(U_\alpha \cap U_i)$. In light of \equ{eq:
  containment}, the latter set contains Bob's choice $B_n$ and so
(applying her strategy for the case that $B_n$ is the first ball
chosen by Bob), Alice wins in this case as well. 
\end{proof}

\smallskip

Note that our definition above depended on the choice of atlas
$\mathcal{A}$. We now deduce from Proposition \ref{prop: does not
  depend on chart} that the HAW property in fact depends only on the
manifold structure on $X$, and not on a specific atlas.   
Namely, recall that two atlases of charts $\mathcal{A}_1 \df (U_\alpha, \varphi_\alpha)$, $ \mathcal{A}_2 \df (V_\beta, \psi_\beta)$ on the same manifold $X$ are said to be {\em $C^1$-compatible} if $\mathcal{A}_1 \cup \mathcal{A}_2$ is also a $C^1$-atlas of in the above sense. 
\begin{cor}\name{cor: does not depend on atlas}
Suppose $\mathcal{A}_1, \mathcal{A}_2$ are two compatible atlases of charts, and $S \subset X$. Then 
$S$ is HAW on $(X, \mathcal{A}_1)$ if and only if it is HAW on $(X, \mathcal{A}_2)$. 
\end{cor}

\begin{proof}
The atlases $\mathcal{A}_1, \mathcal{A}_2$ are compatible if and only if the maximal atlas (with respect to inclusion) $\mathcal{A}_{\max}$ containing $(U_\alpha, \varphi_\alpha)$ coincides with the maximal atlas containing $(V_\beta, \psi_\beta)$. So it is enough to show that $S$ is HAW on $(X, \mathcal{A}_1)$ if and only if it is HAW on $(X, \mathcal{A}_{\max})$. For this, apply  
Proposition \ref{prop: does not depend on chart} with $\mathcal{A}=\mathcal{A}_{\max}$ and $\{(U_i, \varphi_i)\}=\mathcal{A}_1$. 
\end{proof}

In light of Corollary \ref{cor: does not depend on atlas}, we will be
justified below in omitting the atlas from the terminology and  saying
that $S\subset X$ is {\em HAW\/} if it is HAW on $(X, \mathcal{A})$
for some atlas of charts $\mathcal{A}$ defining the manifold structure
on $X$.  
It is clear that Proposition \ref{properties} immediately implies

\begin{prop}\name{manifoldproperties} Let $X$ be a $C^1$ manifold. Then:
\begin{itemize}
\item[(a)]  HAW subsets of $X$ are thick. 
\item[(b)]  The countable intersection of  HAW subsets  of $X$ is again  HAW.
\item  [(c)]
The  image of a HAW 
subset  of $X$ under a $C^1$ diffeomorphism  $X\to X$  is HAW.
\end{itemize}\end{prop}

\smallskip
Now, and for the rest of the paper, let $X$ be as in \equ{defX} and $F$ as in \equ{defF}. 
Also
define $F^+ \df \{g_t : t \ge 0\}$. 
It turns out that Theorem \ref{dyn countable} can be reduced to the
following statement  about one-sided orbits:

 \begin{thm}\name{onesided HAW}  \begin{itemize}
\item[(a)]  The set
\eq{result1}{
\big\{x\in X : F^+x \text{ is bounded}\big \}}
is HAW:
\item[(b)]  
let $Z$ be a compact $(F,H^+)$-transversal  submanifold  of $X$; then the set
\eq{result2}{
\big\{x\in X :  \overline{F^+x}\cap Z  = \varnothing  \big\}}
is  HAW.
\end{itemize} \end{thm}

 \begin{proof}[Proof of Theorem \ref{dyn countable} assuming Theorem
   \ref{onesided HAW}] It is clear that a statement analogous to
   Theorem  \ref{onesided HAW} holds for the semigroup $F^- \df \{g_t
   : t \le 0\}$ in place of $F^+$ and with the roles of $H^+$ and
   $H^-$ exchanged: namely, the sets \eq{result3}{ 
\big\{x\in X : F^-x \text{ is bounded}\big \}}
and 
\eq{result4}{
\big\{x\in X :  \overline{F^-x}\cap Z  = \varnothing  \big\}\,,}
where $Z$ is a   compact $(F,H^-)$-transversal  submanifold  of $X$,
are  HAW.
 The set \eq{result0}{\big\{x\in X : Fx \text{ is bounded and }
  \overline{Fx}\cap Z  = \varnothing  \big\}}
  is the intersection of sets \equ{result1}--\equ{result4}; hence, in
  view of Proposition \ref{manifoldproperties}(b), it is HAW whenever 
   $Z\subset X$ is    compact  and both $(F,H^+)$- and
   $(F,H^-)$-transversal. In Theorem \ref{dyn countable} our set $Z$
   is a countable union of manifolds, and hence (replacing if
   necessary a manifold with a countable union of compact manifolds
   with boundary) a countable union of
   compact manifolds with boundary. Thus the set \equ{result} is the
   countable intersection of sets of the form \equ{result0}, and   
Theorem \ref{dyn countable} follows by another application of
Proposition \ref{manifoldproperties}(b). 
  \end{proof}

 
\ignore{
Reduction of to

 \begin{thm}\name{ba winning}  There exists $\alpha > 0$ such that for any $x\in X$ and any $a\in \R$, the set  \comdima{maybe we should use $\Lambda$ instead of $x$ everywhere}
$$
\big\{g\in G :  a\notin \overline{Q_0(x\nz)} \}$$ 
 is $\alpha$-winning.
 \end{thm}

Likewise, reduction of Theorem \ref{dyn countable} to

 \begin{thm}\name{dyn winning} There exists $\alpha > 0$ such that:
 \begin{itemize}
 \item[\rm (i)] the set
$ \big\{x\in X : Fx \text{ is bounded } \} $  is $\alpha$-winning;
 \item[\rm (ii)]  for any compact  submanifold $Z$ of $X$ which is both $(F,H^+)$-transversal and $(F,H^-)$-transversal, the set \eq{escaping z}{\big\{x\in X :  \overline{Fx}\cap Z  = \varnothing\}} 
 is $\alpha$-winning.
 \end{itemize}
  \end{thm}
\comdima{the logical chain of all the theorems is not clear to me at the moment.}
which, as we show later, implies Theorem \ref{ba winning}. Note that
even part (i) of the Theorem is formally new, it has never been stated
that way (with the game played on $X$).

}

 
 We will finish this section by proving part (a) of the above theorem, which
 is in fact a rather straightforward variation on some well-known
 results. In most of the earlier work concerning winning properties of
 sets of bounded trajectories, the games were actually played on
 expanding leaves 
for the $F^+$-action on $X$,  which in our case can be parametrized as orbits of
the expanding horospherical group $H^+$. 
An example is McMullen's
strengthening \cite[Theorem 1.3]{Mc} of a theorem of Dani
\cite{Dani-rk1} on the winning property of the set of directions in
hyperbolic manifolds of finite volume with bounded geodesic rays, a
special case of which can be restated as follows:

 \begin{thm}\name{bad winning} For any $y\in X$, the set 
\eq{diophset}{\big\{s\in \R :  {F^+h_sy}\text{ is bounded}\}}
 is absolutely winning (which, for games played on $\R$, is the same as HAW).
  \end{thm}

To reduce Theorem
   \ref{onesided HAW}(a) to Theorem
   \ref{bad winning}, let us fix an atlas of coordinate charts for $X$ as follows. Let $\mathfrak g \df \Lie(G)$, 
 and for any $y\in X$ denote by $\exp_y$ the map \eq{defexp}{\exp_y: \g\to X,\quad
 \x \mapsto\exp(\x)
 y\,.} 
 For any $y\in X$ one can choose  a neighborhood $W_y$ of $0\in \g$ 
 such that $\exp_y|_{W_y}$ is
 one-to-one. Denote \eq{defatlas}{U_y \df \exp_y(W_y)\text{ and }\varphi_y \df
 \exp_y^{-1}|_{U_y}\,.} The collection $\{(U_y,\varphi_y): y\in X\}$
 is the atlas that we are going to use.  

 \begin{proof}[Proof of Theorem \ref{onesided HAW}(a)]  
  In view of
Proposition \ref{prop: does not depend on chart} it suffices to show that
for any $y\in X$, the set \eq{bdd orbits}{\varphi_y\big(\big\{x\in U_y
  : F^+x \text{ is bounded}\big \}\big) = \{\x\in W_y:
  F^+\exp(\x)y \text{ is bounded}\big \}} is HAW on $W_y$. We know
from Theorem \ref{bad winning} that the set 
 \equ{diophset}
 is HAW. 
 Note that conjugation by $g_t$, $t\ge 0$, does not expand
 elements of $H^-F$. Therefore, for any $x \in X$,
\eq{stability}{F^+x\text{ is bounded  } \iff\ F^+ gx\text{ is bounded   }\forall\,g\in H^-F\,.} 
The set $U_y$ is foliated by connected components of orbits for the
action of $H^-F$. 
By composing $\varphi_y$ with a suitable diffeomorphism of $W_y$, 
which we are allowed to do by Proposition \ref{properties}(d), we can
assume that this foliation is mapped into the foliation of $\mathfrak{g}$
by translates of $\mathfrak
p \df \Lie(H^-) \oplus \Lie (F)$. Let us denote by $\mathfrak{h}$ the Lie
algebra of $H^+$, by $\pi: \g \to \h$
the 
projection associated with the direct sum
decomposition $\mathfrak g = \mathfrak h \oplus \mathfrak p$, and let $W^+_{y}\df \pi(W_y)$.
\smallskip

 Fix $0 < \beta < 1/3$. Bob begins with
a ball $B_1 \subset W_y$, and Alice consults the strategy
she is assumed to have for playing on $W^+_{y}$  for
the chosen value of $\beta$ and 
taking Bob's first move to be $B'_1 \df \pi(B_1)$. The strategy
specifies an interval (neighborhood of a point)   
$A'_1 \subset \h$, and in the game on $W_y$ Alice chooses $A_1
\df \pi^{-1}(A'_1)$, which is a hyperplane neighborhood. Continuing iteratively, suppose Bob has chosen
the ball $B_i\subset W_y$. The ball $B_i' \df \pi(B_i)$ is a legal move
for Bob playing on $W^+_{y}$, since the projection $\pi$ does not affect the 
radii of balls and since 
$$B_i \subset B_{i-1} \sm A_{i-1} \implies B'_i \subset B'_{i-1} \sm
A'_{i-1}\,.$$
Thus Alice's strategy  for playing on $W^+_{y}$ specifies a move $A'_i
\subset \h$, and in the new game 
Alice chooses $A_i \df \pi^{-1}(A'_i)$. This defines
her strategy for playing on $W_y$ and guarantees that  $\bigcap B'_i$  belongs to  the set 
 \equ{diophset}.  By \equ{stability}, the point
$\bigcap B_i$  belongs to  the set \equ{bdd orbits}. 
 \end{proof}

\section{
Transversality and reduction to discrete time actions
}\name{slicing}
It would seem natural to attempt to prove  an analogue of Theorem
   \ref{bad winning} for orbits escaping $Z$, that is, show that  for $Z$ as in Theorem
\ref{onesided HAW}(b) and $y\in X$, the set 
${\big\{s\in \R :   {F^+h_sy}\cap Z = \varnothing\big\}}
$ is HAW. And indeed the above statement is true; however it would not be
enough for proving  Theorem 
\ref{onesided HAW}(b). A reason for that is that one can state a
version of the equivalence \equ{stability} for this situation, namely,
that  
$$\overline{F^+x}\cap Z = \varnothing \  \iff\ \overline{F^+ \exp(\p)x}\cap Z = \varnothing \ \ \forall\,\p
\in \mathfrak{p} \text{ with }\|\mathfrak{p}\| \le \vre\,,$$
where $\vre> 0$ depends on $x$. However to derive a winning property
of the set \equ{result2} one would need a uniform lower bound on $\vre$,
which is not available here. Thus we will need to play on $X$ itself.  

Our first step will be a reduction to discrete time actions. The argument here loosely follows
\cite[\S4]{K}. For a parameter $\tau>0$ to be defined later, 
denote by
$${F_\tau^+ \df \{g_{n\tau} : n\in\Z_+\} \ \ }
  $$ 
  the cyclic subsemigroup 
of $F$ generated by $g_\tau$. We will make a reduction
showing that we may replace the continuous semigroup $F^+$
with 
 $F^+_\tau$.
For $Z \subset X$ we define 
$$Z_{[t_1, t_2]} \df \bigcup_{t_1 \le t \le t_2}{g_tZ}.$$ 
We have:
\begin{lem}\name{transversality2} Suppose $Z$ is a compact $C^1$ submanifold of $X$ and   $H$ a 
connected subgroup of $G$. Then:

\begin{itemize}
\item[{\rm (a)}] If condition  $(F)$ holds, then there exists  $\sigma = \sigma(Z) > 0$ such that 
$Z_{[-\sigma,\sigma]}$ is a  $C^1$ manifold.

\item[{\rm (b)}]  If, in addition,  condition  $(H,F)$ holds, then there exists positive $\tau = \tau(Z) \le \sigma(Z)$ such that  
for any $x$ in $\zt$, $T_x(Hx)$ is
    not contained in $T_x\left(\zt\right)$.
\end{itemize}
\end{lem}

Note that the conclusion of part (b) of the above lemma coincides with
condition $(F)$  with $F$ replaced by $H$ and $Z$ replaced by $\zt$.  
It will be convenient to introduce more notation and, for a
subgroup $H$  of $G$ and a smooth submanifold $Z$ of $X$, say that
$Z$ is {\sl $H$-transversal at $x\in Z$\/} if  $T_x(Hx)$ is not
contained in $T_xZ$, and that that $Z$ is {\sl $H$-transversal\/}
if it is $H$-transversal at every point of $Z$. In other words, if
condition $(F)$ holds with $F$ replaced by $H$.  This way, condition
$(F)$ says that $Z$ is $F$-transversal, and  the  conclusion of Lemma
\ref{transversality2}(b) states that  $\zt$ is $H$-transversal.

 \begin{proof}
Let $\dim(Z) = k$. Using a finite covering of $Z$ by appropriate
coordinate charts of $X$, one can without loss of generality assume
that $Z$ is   of the form  $\psi\left(\overline U\right)$ for some
bounded open $U\subset \br^k$, with $\psi$ being a $C^1$, nonsingular
embedding defined on an open   $U'\subset \R^k$ strictly containing
$\overline U$. Define $\tilde \psi:U'\times\R\to X$ by putting $\tilde
\psi (u,t) = g_t\big(\psi(u)\big)$. From  the $F$-transversality of
$Z$ it follows that $\tilde\psi$ is nonsingular at $t=0$ and
$u\in\overline U$. Hence  $\tilde\psi$ is a nonsingular embedding of
$U''\times [-\sigma, \sigma]$ into $X$ for some   $\sigma>0$ and  an
open set $U''$ strictly containing $\overline U$, and (a) is proved. 

Clearly the tangent space to $Z_{[-\sigma,\sigma]}$ at  $x\in Z$ is equal to $T_xZ\oplus T_x(Fx)$. 
Therefore condition $(H,F)$ implies that $Z_{[-\sigma,\sigma]}$ is
$H$-transversal at any point of $Z$. But the $H$-transversality is
clearly an open condition, hence  it holds at any point  of
$Z_{[-\sigma,\sigma]}$ which is close enough to $Z$. By compactness one can
choose a positive $\tau\le \sigma$ such that  $\zt$ is  $H$-transversal.  
 \end{proof}

Here is another way to express $H$-transversality: fix a  Riemannian
metric `dist' on the tangent bundle of
$X$, and for a $C^1$
submanifold $Z$ of $X$ consider the function $\thz:Z\to\R$,   
$$
\thz(x)\df \sup_{v\in T_x(Hx),\,\|v\| = 1}\dist (v,T_xZ).
$$
It is clear that 
$\thz(x)\ne 0$ iff $Z$ is  $H$-transversal at $x$, and that  $\thz$ is
continuous in $x\in Z$. Therefore the following holds:  

\begin{lem}\name{transversality function} 
A compact $C^1$ submanifold $Z$ of $X$ is
$H$-transversal iff there exist  $c = c(Z) > 0$ such that $\thz(x) \ge c$ for all $x\in Z$. 
\end{lem}
A right-invariant metric on $G$ induces a well-defined Riemannian
metric on $X$ and we now change notation slightly, writing `dist' for
the resulting path metric on $X$.
Now let $Z$ be as in Theorem \ref{onesided HAW}, that is, compact and
$(H^+,F)$-transversal, and take positive $\tau\le \tau(Z)$ as in  Lemma
\ref{transversality2} satisfying in addition: 
\eq{lipschitz}{x,y\in X, 0\le t \le \tau\ \Rightarrow\ \dist(g_tx,g_ty) \le 2 \dist(x,y)\,.}
This can be done because $g_t, |t|\leq \tau$ is bounded and hence
there is a uniform bound on the amount by which it distorts the
Riemannian metric. 
Suppose that for some $x\in X$ and $\vre > 0$
there exists $t\ge 0$ and $z\in Z$ such that the distance between
$g_tx$ and $z$ is less than  $\vre$. Choose $0 \le t_1 < \tau$ such
that $t+t_1 = n\tau$ for some $n\in \N$.  
It follows from \equ{lipschitz} that $\dist(g_{n\tau}x,g_{t_1}z) <
2\vre$.
This rather
elementary argument proves that 
$$ \overline{F^+_\tau x}\cap \zt
= \varnothing\ \Rightarrow\  \overline{F^+
  x}\cap Z
  = \varnothing\,. 
$$
Thus to establish Theorem \ref{onesided HAW}(b) it suffices to prove the following:

  \begin{thm}\name{h discrete winning} For any compact
    $H^+$-transversal  submanifold $Z$ of $X$ and any  $\tau > 0$   
    the set  
\eq{discrete goodset}{\big\{x\in X :   {F^+_\tau x}\cap Z 
  = \varnothing\big\}} 
 is  HAW.
  \end{thm}
 \ignore{, and meanwhile let us 
point out that using it one can derive a discrete-time version of
Theorems \ref{dyn countable} and \ref{onesided HAW}:\comdima{Should we
  advertise it as one of the main results? BW: recall I actually think
we can omit it.}

 \begin{cor}\name{dyn discrete countable}  Fix $\tau > 0$. 
 \begin{itemize}
 \item[(i)] Let $Z$ be a  countable union of   $H^+$-transversal submanifolds  of $X$. Then  the set
$$
\big\{x\in X : F_\tau^+x \text{ is bounded and } \overline{F_\tau^+x}\cap Z  = \varnothing\}$$
is  HAW, hence thick;

 \item[(ii)] Let $Z$ be a  countable union of    submanifolds  of $X$ which are both $H^+$- and $H^-$-transversal. Then  the set
$$
\big\{x\in X : F_\tau x \text{ is bounded and } \overline{F_\tau x}\cap Z  = \varnothing\}$$
is  HAW, hence thick.
 \end{itemize} \end{cor}}

\section{Completion of the proof: the percentage game
}
\name{sec: completing proof}
In this final section we prove Theorem \ref{h discrete winning}.  
The argument below originates with an idea of Moshchevitin \cite{Mo},
which was developed further in the Ph.D.\ Thesis of Broderick, and in
the papers \cite{B, BFK, BFS}. We
follow the streamlined presentation of
\cite{BFS}, which consists of defining yet another game. 

Fix $  \beta > 0$  and a target set $S\subset \R^d$.
The {\em hyperplane percentage game} is defined as follows: Bob begins as 
usual by choosing a closed ball 
$B_1 \subset \R^d$. Then, for each $i \ge 1$, once $B_i$ (of radius $r_i$)
is chosen, Alice chooses finitely many affine hyperplanes $L_{i,j}$ and
numbers $\vre_{i,j}$, where $j=1,\dots,N_i$,
satisfying $0 < \vre_{i,j} \le \beta r_i$.
Here $N_i$ can be any positive integer that Alice chooses. Bob then
chooses a ball $B_{i+1}\subset B_i$ with radius   $r_{i+1} \ge 
\beta r_i$ such that $$B_{i+1} \cap  
L_{i,j}^{(\vre_{i,j})} = \varnothing\text{ for at least $\frac{N_i}{2}$ values of }j\,.$$
Thus we obtain
as before a nested sequence of closed balls $B_1 \supset B_2 \supset
\cdots$  and declare Alice the winner if and 
only if \equ{Alicewins} holds.  If Alice has a strategy to win
regardless of Bob's play, we say that $S$ is 
\textsl{$\beta$-hyperplane percentage winning, or $\beta$-HPW}. 
Note that for large values of $\beta$ it is possible for Alice to
leave Bob with no available moves
after finitely many turns.   However, an elementary argument
(see \cite[Lemma 2]{Mo} or \cite[\S2]{BFK}) shows  that Bob always has
a legal move if $\beta$ is smaller 
than some constant $\beta_0(d)$. In particular $\beta_0(1) =
1/5$.
If $S$ is $\beta$-HPW for each $ 0 < \beta < \beta_0(d)$, we say that
$S$ is \textsl{hyperplane percentage winning (HPW)}. 
It is clear that $\beta$-HPW
implies  $\beta'$-HPW if $\beta \le \beta'$; thus HPW is equivalent to
$\beta$-HPW for arbitrary small values of $\beta$.  

We remark that the game defined above is actually a special case
 of the $(\beta,p)$-game defined in \cite{BFS}, corresponding to the choice
 $p=1/2$. Here $p$ represents the percentage of hyperplanes that Bob
 is obliged to stay away from. 

One sees that in the hyperplane percentage game the rules are more
favorable to  Alice than in the hyperplane absolute game, and so any
HAW set is automatically HPW. Surprisingly, the converse is true, see
\cite[Lemma 2.1]{BFS}: 

 \begin{lem}\name{percentage} For any $\beta \in (0, 1/3)$ there
   exists $\beta' \in \big(0, \beta_0(d)\big)$ such that any set which is
   $\beta'$-HPW is
   $\beta$-HAW. In particular the HPW and HAW properties are equivalent. 
\end{lem}

\ignore{{\comment{DK:  \cite[Lemma 2.1]{BFS} did not specify the relation
    between $\beta'$ and $\beta$. Should we do it and explicitly write
    down a function $\beta'(\beta,d)$? what do you think? the paper is
    at http://arxiv.org/pdf/1208.2091.pdf 

BW: I was too lazy to figure this out. I suggest we leave this as is. 
}} 
}

 \begin{proof}[Proof of Theorem \ref{h discrete winning}]  Recall that
   we are given $\tau,\beta > 0$ and a compact 
    $H^+$-transversal  submanifold $Z$ of $X$. Our goal (after using
    Lemma \ref{percentage} and rewriting $\beta$ for $\beta'$) is to 
    show that the set \equ{discrete
      goodset} is $\beta$-HPW. We will assume, without loss of
    generality, that  
    \eq{conditiononbeta}{\beta < e^{-2\tau}\,.}
     Let us say that a map between two metric spaces is {\sl
    $C$-bi-Lipschitz\/} if the ratio of ${\dist\big(f(x),f(y)\big)}$
  and ${\dist(x,y})$ is uniformly bounded between $1/C$ and $C$.  Also, for a subset $Y$ in a metric space and $\vre > 0$  let us denote by $Y^{(\vre)}$ the $\vre$-neighborhood of $Y$. 

\smallskip

The first step is to fix an atlas of coordinate charts for $X$. 
We will do it as in the proof of Theorem \ref{onesided HAW}(a), that is, using 
charts  \equ{defatlas} where $\exp_y$ is defined by  \equ{defexp} and $\exp_y|_{W_y}$ is
 one-to-one for all $y\in X$. 
As before, it suffices to show that
for any $y\in X$, the set 
$$\varphi_y\big(\big\{x\in U_y
  :  F^+_\tau x\cap Z 
  = \varnothing\big \}\big) = \{\x \in W_y:
  F^+_\tau\exp(\x)y \cap Z 
  = \varnothing\big \}$$
   is HAW on $W_y$.
  
\ignore{Recall that in \equ{defexp} we defined a map $\exp_y:\g\to X$ which
sent cosets of $\mathfrak p$ to orbits of the group  
$FH^-$. 

In this proof the main role is played by 
$H^+$-orbits and it will be more convenient to use another
map. Namely, for any $y\in X$ denote
by $\Exp_y$ the map 
\eq{defExp}{\Exp_y: \g\to X,\quad 
 \x \mapsto\exp(s)\exp(\p)y = h_s\exp(\p)y\,,} 
 where $s\in\mathfrak h\cong\R, \ \p\in \mathfrak p\cong \R^2$ and $(s, \p)
 \in \mathfrak{h} \oplus \mathfrak{p} \cong \mathfrak{g}.$ }

\smallskip
 
 The next step is to collect some information about $Z$. Since it is
 bounded, one can choose $\sigma_1 < 1$ such that the restriction of
 $\exp_y$ to $B_\g(0,4\sigma_1)$ is $2$-bi-Lipschitz (in particular,
 injective) whenever   $B_X(y,2\sigma_1)\cap Z \ne\varnothing$ 
(here, as before, we work with a path-metric on $X$ coming from a
right-invariant 
Riemannian metric on $G$ obtained from some inner product on
$\g$). In light of Corollary \ref{cor: does not depend on atlas}, we
can (by replacing the $U_y$ with smaller sets, depending on $Z$) also assume that
$U_y\subset B_X(y,2\sigma_1)$ for any 
$z\in Z^{(4\sigma_1)}$. 
Then, because of compactness and $C^1$-smoothness of $Z$, for every $b > 0$
there exists $0 < \sigma_2(b)\le \sigma_1$ such that
\eq{smoothness}{
\begin{aligned}
\sigma \le \sigma_2(b) \text{ and }   B_X(y,\sigma)\cap Z\ne\varnothing
\quad\Rightarrow\quad \text{there exists}\\   \text{ a $\dim(Z)$-dimensional subspace $L$ of }\g\text{ such that } 
\varphi_y\big(B_X(y,\sigma)\cap Z&\big) \subset L^{(b\sigma)}\,.\end{aligned}} 
(More precisely, $L$ is the tangent space to $\varphi_y(Z)$  at
$\varphi_y(z)$, where $z$ is some point in the intersection of   $
B_X(y,\sigma)$ and $Z$.) Also, recall that $Z$ is 
$H^+$-transversal, and let $c = c(Z)$ be as in Lemma \ref{transversality function}.

Now  choose $m$ large enough so that
\eq{define r}{\beta^{- n} < e^{2m\tau},\quad 
\text{ where } n = \lfloor\log_2 m\rfloor + 1\,.}
This is possible since the left (resp., right) side of the inequality
in \equ{define r} depends polynomially (resp., exponentially) on $m$.

Pick an arbitrary $y\in X$ and let us suppose that Bob chooses a ball $B_0 \subset W_y$ of radius $r_0$.
Set 
\eq{set1}{
b\df \frac {c\beta^{ n+2}e^{-2m\tau}}{16}\,,
}
\eq{set3}{
\sigma \df \min\left(\frac14\sigma_2(b),e^{2m\tau}r_0\right)\,,} 
and let  
\eq{set2}{\delta \df e^{-2m\tau}\sigma \quad\text{ and 
}\quad\vre \df b\sigma =  \frac1{16}c\beta^{ n+2}\delta\,. } 
We will show that Alice can play the $\beta$-hyperplane percentage game in such a way
that the intersection of all the balls belongs to the set 
$$ \{\x \in W_y: 
  F^+_\tau\exp(\x)y \cap Z^{(\vre)} 
  = \varnothing\big \}\,.$$

Note that it follows from \equ{set3} and \equ{set2} that $\delta \le
r_0$. The game will start with Alice making dummy moves until the
first time Bob's ball has 
radius $r_1 \le \delta$ (recall that if the radii $r_j$ of the $B_j$ do not tend to
zero, then Alice wins, see Remark \ref{remark: radii}). Re-indexing if
necessary, let us call this ball $B_1$ and its radius $r_1$; note that 
we have  $r_1 \ge \beta\delta$.  

In order to specify Alice's strategy we will partition her moves into
windows. For each $j, 
k \in \N$, we will say that {\em $k$ lies in the $j$th window} if
\eq{jthwindow} {\beta^{- n(j-1)}\le e^{2k\tau} <  \beta^{-j n}\,.}
By \equ{define r}, for every $j  \in \N$, there are at most $m$ indices $k$ lying in the $j$th window, and \equ{conditiononbeta} guarantees that every $k\in \N$ lies in some window.
We will call the
indices $i$ for which 
\eq{jthsize}{\beta^{ n j}r_1 < r_i \le
  \beta^{ n(j-1)}r_1}
the {\em $j$th stage of the game.} 
The first stage begins with Bob's initial ball $B_1$, and the rules of the game
imply that each stage contains 
at least $ n$ indices. Loosely speaking, Alice will use her moves
indexed by numbers in
the $j$th stage, to ensure that the points $\x $ contained in the balls
chosen by Bob  satisfy $g_{ \tau}^k\exp(\x) y \notin Z^{(\vre)}$  for any $k$ in the
$j$th window. We now specify Alice's strategy in more detail.

Fix $j$ and suppose that $i = i(j)$ is the first index of stage
$j$.
For any $k$ belonging to the $j$th window, 
denote
$$A_{j,k} \df \{g_\tau^k\exp(\x) y : \x  \in B_i \} 
\,.$$
In view of \equ{jthsize}, the diameter of $A_{j,k}$ is at most 
$$2e^{2k\tau}\beta^{ n(j-1)}r_1 \stackrel{\equ{jthwindow}}{\le} 
2\beta^{- n}r_1 \stackrel{\equ{define r}}{\le} 2 e^{2m\tau}r_1
\stackrel{\equ{set2}}{\le} 2\sigma \stackrel{\equ{set3}}{\le} \frac12 
\sigma_2(b)\,.$$
 If $A_{j,k}$ does not intersect $Z^{(\vre)}$  for any $k$ in the
 $j$th window, Alice makes all her moves in the $j$th stage of the game 
 in an arbitrary way (e.g.\ she could put $N_i=0$, that is decide not to specify
 any hyperplanes on her $i$th move). Suppose that one of them does; let $x \df g_\tau^ky$. Then 
$A_{j,k}\cap B_X\big(x,  \sigma_2(b)\big)\ne\varnothing$. Now let
 us use $\varphi_x$ to map everything to $\g$. In view of
 \equ{smoothness},  there exists $z\in Z$ and  a $\dim(Z)$-dimensional subspace $L = T_{\varphi_x(z)}\big(\varphi_x(Z)\big)$ of $\g$ such that
 \begin{equation*}
 \begin{split}
 \varphi_x\left(A_{j,k} \cap Z^{(\vre)}\right) \subset  \varphi_x\left(A_{j,k} \right) \cap \left( \varphi_xZ\right)^{(2\vre)}
\stackrel{\equ{smoothness}}{\subset}  \varphi_x\left(A_{j,k} \right) \cap L^{(2\vre + 2b\sigma)}\\
\stackrel{\equ{set1},\, \equ{set2}}{\subset}  \varphi_x\left(A_{j,k}
\right) \cap L^{\left(c\beta^{ n+2}e^{-2m\tau}\sigma/4\right)}\,. 
 \end{split}
 \end{equation*}
In view of  Lemma \ref{transversality function} and the $2$-bi-Lipschitz property of maps $\varphi_x$ and $\varphi_z$, the intersection of $L^{\left(c\beta^{ n+2}e^{-2m\tau}\sigma/8\right)}$ with any translate of $\h$ has length at most $2\beta^{ n+2}e^{-2m\tau}\sigma$. 
Consequently, the intersection of 
\eq{kthset}{
\left \{ \x  \in B_i : g_\tau^k\exp(\x) y \in Z^{(\vre)} \right \}  =
\varphi_y\left(g_\tau^{-k}(A_{j,k} \cap Z^{(\vre)})\right)   =
\Ad_{g_\tau^{-k}}\left(\varphi_x(A_{j,k} \cap Z^{(\vre)})\right)}  
 with any translate of $\h$ has length at most 
 $$2e^{-2k\tau} \beta^{ n+2}e^{-2m\tau}\sigma \stackrel{\equ{jthwindow},\, \equ{set2}}{\le}
2\beta^{ n(j-1)}\beta^{ n+2}\delta 
{\le}
2\beta^{ n(j-1)}\beta^{ n+1}r_1 \le2 \beta^{ n}r_i
\,.$$ 
This implies that $
 \left\{ \x  \in B_i : g_\tau^k\exp(\x) y \in Z^{(\vre)} \right\}$ is
 contained in a $ \beta^{ n}r_i$-neighborhood of some hyperplane, and
 hence 
 the union of sets \equ{kthset} over all indices $k$
contained in the $j$th
window   is covered by at most $m$ $ \beta^{ n}r_i$-neighborhoods of hyperplanes.
Let us denote these hyperplanes by $L_{i,j}$ and let $\vre_{i,j} = \beta^{ n}r_i$. 
These will be Alice's choices in the $i$th move;
the above discussion ensures that they constitute a valid move for
Alice. In each of the remaining steps belonging
to the $j$th stage, Alice will choose those neighborhoods which still
intersect the ball chosen by Bob. That is, if we write the indices
belonging to the $j$th stage as $ i(j), i(j)+1, \ldots,
i(j+1)-1$, Alice's choices in stage $j$ will be 
those hyperplanes $L_{i,j}$ for which $L_{i,j}^{(\vre_{i,j})} \cap
B_\ell \neq \varnothing$, equipped with $\vre_{i, j} = \beta^{ \ell}r_{i(j)}$. 
This choice and \equ{jthsize} ensure that all of these moves are valid moves
for Alice. 

Since the $j$th stage contains at least $ n$ indices, and in
every one of his moves Bob must choose a ball intersecting at most
$1/2$ of the intervals chosen by Alice, we have that out of those 
neighborhoods, at most $2^{- n}m$ can intersect the first ball $B_{i(j+1)}$
in the first move of the $(j+1)$-th stage of the game. But
$2^{- n}m<1$  in view of \equ{define r}. Consequently, for $k$ in the
$j$th window and $\x\in B_{i(j+1)}$, we have that $g_\tau^k\exp(\x)y$ is not
in $Z^{(\vre)}$, and therefore, if $\x_{\infty}$ is  the intersection point
of all Bob's balls, it follow that  $\exp(\x_\infty)y$ belongs to the set \equ{discrete
  goodset}. The theorem is 
proved.
 \end{proof}

\bibliographystyle{alpha}

\end{document}